\documentclass[a4paper, leqno]{amsart}
\usepackage{amsmath,amsthm}
\usepackage{amsfonts}
\usepackage{color}
\usepackage{amssymb}
\usepackage{graphicx}
\usepackage[all]{xy}
\usepackage{mathrsfs}

\newtheorem{theorem}{Theorem}[section]

\newtheorem{proposition}[theorem]{Proposition}
\newtheorem{lemma}[theorem]{Lemma}
\newtheorem{corollary}[theorem]{Corollary}

\newtheorem{definition}[theorem]{Definition}

\newtheorem{example}[theorem]{Example}
\newtheorem{remark}[theorem]{Remark}
\newtheorem{theoremz}{Theorem}

\newcommand{\rmap}{\to}

\newcommand{\X}{\ensuremath{\mathfrak{X}}}

\newcommand{\pr}{pr}

\renewcommand{\hom}{\mathrm{Hom}}
\DeclareMathOperator{\modular}{mod}     
\renewcommand{\mod}{\modular}

\newcommand{\G}{\mathcal{G}}            

\renewcommand{\H}{\mathcal{H}}          

\newcommand{\Lie}{L}          

\renewcommand{\rmap}{\to}

\newcommand{\al}{\alpha}

\newcommand{\rank}{\text{\rm rk}\,}   

\title{Jacobi structures and Spencer operators}
\author{Marius Crainic}
\address{Mathematical institute, Utrecht University, Utrecht, The Netherlands}
\email{M.Crainic@uu.nl}

\author{Mar\'\i a Amelia Salazar}
\address{Max-Planck-Institut f\"ur Mathematik, Bonn, Germany}
\email{mariasalazarp@gmail.com}

\thanks{This research was financially supported by the ERC Starting
Grant no. 279729.}
\date{\today}

\begin{document}
\maketitle

\begin{abstract}
This paper reveals the fundamental relation between Jacobi structures and the classical Spencer operator coming from the theory of PDEs \cite{Spencer,KumperaSpencer}; in particular, we provide a direct and much simpler/geometric approach to the integrability of Jacobi structures. It uses recent results on the integrability of Spencer operators and multliplicative forms on Lie groupoids with non-trivial coefficients \cite{Maria, thesis}. \end{abstract}

\section{Introduction}

In this paper we provide a new approach to the integrability of the Jacobi structures of Lichnerowicz \cite{Lich78} and the local Lie algebras of Kirillov  \cite{Kirillov};
it is based on our remark that Jacobi structures are intimately related to the classical Spencer operator coming from the geometric theory of PDEs \cite{Spencer,KumperaSpencer}, combined with
our recent result on the integrability of Spencer operators and multiplicative forms with coefficients \cite{Maria, thesis}. This approach is not only new, but also much more direct/geometric and remarkably simpler than the known ones (see the long formulas from \cite{Jacobi}). In this introduction we describe the main key-words and literature that come with Jacobi structures, indicating along the way the content of this paper. \\

\noindent\textbf{Lichnerowicz's Jacobi structures:} Jacobi structures were discovered by Lichnerowicz who, after his work on Poisson and symplectic structures, was looking for a similar theory in which the symplectic structures were replaced by their odd-dimensional analogue, i.e. contact structures. He introduced them 
in \cite{Lich78} and then studied them intensively \cite{Lich83,Lichnerowicz,Dazord-Lich-Marle91}, etc. In Lichnerowicz's terminology, a Jacobi structure is a pair 
$(\Lambda, E)$ consisting of a bivector $\Lambda$ and a vector field $R$ on $M$, satisfying certain first order differential equations:
\begin{equation}
\label{first-order} 
[\Lambda, R]= 0, \ [\Lambda, \Lambda]= 2R\wedge \Lambda 
\end{equation}
(see also below). Lichnerowicz also studied a locally conformal version of the theory, in which the pairs $(\Lambda, R)$ are defined only locally
and, on the overlaps, they are related by certain (conformal) transformations \cite{Lich78}; in particular, Lichnerowicz's locally conformal Jacobi structures
come with an underlying line bundle arising from the transition functions on the overlaps. This aspect was further clarified by Marle \cite{Marle} who 
uses the term {\it Jacobi bundles} for the resulting line bundles. \\

\noindent\textbf{Kirillov's local Lie algebras:} Interesting enough, and very much relevant to the present paper, Lichnerowicz's Jacobi structures turned out to be ``essentially the same'' as
the local Lie algebras structures (on line bundles) that were considered by Kirillov already a few years earlier \cite{Kirillov}. More precisely, Kirillov studied Lie algebra structures
\[ \{\cdot, \cdot\}: \Gamma(L)\times \Gamma(L)\rmap \Gamma(L) \]
on the space $\Gamma(L)$ of sections of a line bundle $L$ over a manifold $M$, which are local in the sense that, for $u, v\in \Gamma(L)$ supported in some open $U\subset M$, $\{u, v\}$ is supported in $U$ as well. Only later Lichnerowicz remarked \cite{Lich83} that
\begin{itemize}
\item the previous equations (\ref{first-order}) are equivalent to the condition that the bracket $ \{\cdot, \cdot\}$ defined on the space $C^{\infty}(M)$ of smooth functions on $M$ by:
\begin{equation}\label{Jacobi-br} 
\{f, g\}= \langle df\wedge dg, \Lambda\rangle + fL_R(g)- gL_R(f)  
\end{equation}
satisfies the Jacobi identity (where $\langle \cdot, \cdot\rangle$ denotes the pairing between forms and multi-vectors, and $L_R$ denotes the Lie derivative along $R$).
\item Kirillov's work actually shows that any Lie algebra structure on $C^{\infty}(M)$ which is local must be of this type.
\end{itemize}
In other words, Lichnerowicz's Jacobi structures $(\Lambda, R)$ are the same thing as Kirillov's local Lie algebras with trivial underlying line bundle. Moreover, this extends to
the case of arbitrary line bundles, with the conclusion that Lichnerowicz's locally conformal Jacobi structures are the same thing as Kirillov's local Lie algebras.\\

\noindent\textbf{The line bundle:} We will adopt the following terminology: pairs $(\Lambda, R)$ as above will be called here {\it Jacobi pairs} or a {\it Jacobi structure on the trivial line bundle}, while the term {\it Jacobi structure} will be reserved for 
the resulting locally conformal theory or, equivalently, for local Lie algebra structures on an arbitrary line bundle. Hence Jacobi pairs correspond to Jacobi structures with 
trivial underlying line bundle. Note that this corresponds to the similar terminology from Contact Geometry, where one talks about contact forms and contact structures on a manifold $M$: the later are encoded in contact hyperplanes $\mathcal{H}\subset TM$ and come together with the normal line bundle $L= TM/\mathcal{H}$; contact forms correspond to the
case when $L$ is the trivial line bundle. Of course, this is more than an analogy since, as it is already clear from the original work of Lichnerowicz, we know that 
\begin{itemize}
\item contact forms are in 1-1 correspondence with non-degenerte Jacobi pairs $(\Lambda, R)$ (where non-degeneracy means $TM= \Lambda^{\sharp}(T^*M)\oplus \mathbb{R}\cdot R$, $\Lambda^{\sharp}$ is $\Lambda$ interpreted as a linear map $T^*M\rmap TM$)
\item similarly (and more generally), contact structures are in 1-1 correspondence with non-degenerate Jacobi structures.
\end{itemize}
Since in the contact case one very often makes the (rather mild) assumption that the line bundle is trivial, we would like to emphasize here that one of the 
points of this paper is that, for general Jacobi structures, it is important to allow and work with general line bundles. There are at least three reasons:
\begin{itemize}
\item the resulting arguments are much more geometric and less computational (in particular, we invite the reader to compare this paper with \cite{Jacobi}). 
\item the line bundle carries an extra structure (that of a representation) and, even when the line bundle is trivial as a vector bundle, the extra-structure is almost never trivial (see the comments of Examples \ref{trivial-but-not-1} and \ref{trivial-but-not-2}). 
\item there are interesting contact structures with non-trivial normal line bundle, for instance the manifold of contact elements on an affine space (for which the name {\it contact structure} is due) \cite{Arnold}.
\end{itemize}

\noindent\textbf{Poissonization:}  Similar to the ``symplectization of a contact manifold''\cite{contact}, and generalizing it, one can talk about the Poissonization of Jacobi pairs \cite{Lich78,Dazord-Lich-Marle91}, obtaining a 1-1 correspondence between Jacobi pairs $(\Lambda, E)$ on $M$ and homogeneous Poisson structures on $M\times \mathbb{R}$. This gives rise to the so-called ``Poissonization trick'' for proving results about Jacobi pairs, by moving to the Poisson world. However, we would like to point out that this is very unsatisfying. On one hand, the resulting arguments  are mainly algebraic, lacking in geometric insight. On the other hand, since one works explicitly with Jacobi pairs, the resulting arguments are not only algebraic but also very computational (because of the reasons mentioned above). One of the points we are trying to make in this paper is that, by paying enough attention to the line bundle and its structure (the relationship with the Spencer operator), the arguments become much more direct, geometrical, and conceptual (in particular, free of unnecessary computations). \\

\noindent\textbf{Integrability:} We now return to our historical comments on Jacobi structures. Following 
\begin{itemize}
\item Lichnerowicz's philosophy that the interaction between Jacobi structures and contact structures is analogous to the one between Poisson and symplectic structures,
\item the fact that the global objects underlying (or better: integrating) Poisson structures are the symplectic groupoids (i.e. Lie groupoids endowed with a
symplectic form ``compatible'' with the groupoid multiplication) 
\end{itemize}
it was expected that there were a notion of ``contact groupoid'' that integrates Jacobi structures. This problem was first solved in the case when the underlying bundle was trivial  \cite{Kerbrat,Libermann93}. 
The outcome seemed, at least at first sight, very un-natural and certainly unaesthetic  (see \cite{Kerbrat}). The reason is the same as above: while this case is apparently (!!!) simpler, the structure involved
is not visible, as the line bundle, although trivial as a vector bundle, is not trivial as a representation, giving rise to a certain mysterious cocycle.

The general case was considered by Dazord in \cite{Dazord1} and it turned out to be much more elegant (geometric and less computational). It is interesting to point out that Dazord's motivation for looking at the integrability of Jacobi structures was very much related to 
Kirillov's point of view: they provide an intermediate step in the process of integrating the local Lie algebra to a Lie group; indeed, with the contact groupoid at hand, there is a natural notion of Legendrian bisections of the groupoid and they form the desired Lie group. \\

\noindent\textbf{Jacobi structures and contact groupoids:}
We now have a closer look at the process of integrating a Jacobi structure $(L, \{\cdot, \cdot\})$ on a manifold $M$ to a contact groupoid $(\Sigma, \mathcal{H})$ (where $\Sigma$ denotes the Lie groupoid and $\mathcal{H}$ the contact hyperplane). The resulting story is
completely similar to that from Poisson Geometry: 
\begin{itemize}
\item for any contact groupoid $(\Sigma, \mathcal{H})$ over $M$ there is an induced Jacobi structure $(L, \{\cdot, \cdot\})$ on $M$. Moreover, the infinitesimal counterpart of $\Sigma$, i.e. its Lie algebroid, depends only on the Jacobi structure: it is the first jet bundle $J^1L$ of $L$ with the Lie algebroid bracket expressed in terms of the bracket $ \{\cdot, \cdot\}$ of $L$. 
\item conversely, starting with a Jacobi structure $(L, \{\cdot, \cdot\})$, one can talk about the associated Lie algebroid $J^1L$ \cite{Dazord}. Hence, to build $(\Sigma, \mathcal{H})$,  one first has to assume that the algebroid $J^1L$ is integrable. 
\item the integrability problem for the given Jacobi structure becomes: if $J^1L$ is integrable by a Lie groupoid $\Sigma$, is there a contact structure $\mathcal{H}$ on $\Sigma$, making $(\Sigma, \mathcal{H})$ into a contact groupoid for which the induced Jacobi structure on the base is the original one? 
\end{itemize}

\noindent The proofs of such results are spread over the literature. The most difficult part (the question above) was treated in  \cite{Jacobi}. However, most of the arguments (in particular the entire  \cite{Jacobi}) are based on the ``Poissonization trick''; they are based on long computations and lack geometric insight. This paper provides the direct approach.\\

\noindent\textbf{Spencer operators:} We now move to the second part of our title. The classical Spencer operator associated to a vector bundle $E$ over a manifold $M$ \cite{Spencer} is the operator
\[ D: \Gamma(J^1E)\rmap \Omega^1(M, E)\]
 which controls the sections of $J^1E$ which are holonomic, i.e. of type $j^1(s)$ for some $s\in \Gamma(E)$: they are those sections that are zeroes of $D$. 
The Spencer operator that is relevant to this paper is simply the one associated to $E= L$ -the line bundle underlying a Jacobi structure. Of course, it is not just the operator $D$ that is important, but also the structure that
it comes (and interacts) with- structure that reflects the fact that we deal with a Jacobi structure and not just with a line bundle.

This brings us to the notion of Spencer operators. These make sense as soon as we fix an algebroid $A$ and a representation $E$ of $A$; they are operators 
\[ D: \Gamma(A)\rmap \Omega^1(M, E)\]
with the same properties as the classical Spencer operator and are compatible with the Lie brackets involved. If the Lie algebroid $A$ comes from a Lie groupoid $\Sigma$, such Spencer operators are the infinitesimal counterpart of 
1-forms on $\Sigma$ with coefficients in $E$, which are compatible with the multiplication (they are {\it multiplicative}); one of the main results of \cite{Maria, thesis} proves an integrability theorem in this context.

In summary, the main steps for the integration of a Jacobi structure $(L,\{\cdot,\cdot\})$ are: consider the Lie algebroid $J^1L$, note that the classical Spencer operator $D$ is compatible with the brackets, consider the multiplicative one form integrating $D$ (on the groupoid $\Sigma$ integrating $J^1L$) and take its kernel. \\

\noindent\textbf{The content of this paper:} In Section \ref{Contact manifolds and their brackets} we review some of the basic notions on contact structures, including the (probably not so well-known) associated Jacobi bracket. Section \ref{jacobi} is devoted to Jacobi structures and the associated Lie algebroids. Section 4 indicates the relevance of Spencer operators in the theory of Jacobi structures and discusses its global counterpart (multiplicative forms and distributions). Section \ref{Contact groupoids} recalls and discusses contact groupoids. Section \ref{From contact groupoids to Jacobi manifolds} uses Spencer operators to show that the base of a contact groupoid carries an induced Jacobi structure (Theorem \ref{theorem-1}). Finally, Section \ref{sec:jactoctc} takes the reverse problem of integrating a Jacobi structure to a contact groupoid (Theorem \ref{theorem-2}).

\section{Contact manifolds and their brackets}
\label{Contact manifolds and their brackets}

This section recalls some basic notions on contact manifolds.

A {\bf contact structure} (or hyperplane) on a manifold $M$ is a hyperplane distribution $\H\subset TM$ which is maximally non-integrable, i.e. it has the property that the {\it curvature} \begin{equation}\label{eq:curvature}c_{\H}:\H\times \H\to L\end{equation} is non-degenerate. Here, $L$ is the quotient line bundle
\[ L:= TM/\H\]
and $c_{\H}$ is given at the level of sections by $c_{\H}(X, Y)= [X, Y]$ mod $\H$. 

\begin{definition} A Reeb vector field of the contact manifold $(M, \H)$ is any vector field $R$ on $M$ such that
\[ [R, \Gamma(\H) ]\subset \Gamma(\H).\]
We denote by $\X_{\mathrm{Reeb}}(M, \H)$ the set of Reeb vector fields.
\end{definition}
 
The notion of Reeb vector field also appears in the literature under the name of {\it contact vector field} (e.g. \cite{Audin}).

\begin{lemma}\label{lemma-1} $\X_{\mathrm{Reeb}}(M, \H)$ is a Lie subalgebra of the Lie algebra $\X(M)$ of all vector fields on $M$ and
\[ \X(M)= \X_{\mathrm{Reeb}}(M, \H)\oplus \Gamma(\H).\]
\end{lemma}

\begin{proof} The first statement follows from the Jacobi identity for the standard Lie bracket of vector fields. For the second part, if $ \X(M)= \X_{\mathrm{Reeb}}(M, \H) + \Gamma(\H)$, then non-degeneracy of $c_{\H}$ implies that the sum is direct. Let now $X\in \X(M)$ be arbitrary. Consider the map
\[ [X, -]: \H\rmap L, \ \ W\mapsto [X, W] \ \textrm{mod}\ \H \]
(a priori, the above formula is defined only on sections, but it is easily seen to be $C^{\infty}(M)$-linear). Non-degeneracy of $c_{\H}$ implies that there exists $V\in \H$ such that this map coincides with $c_{\H}(V, -)$. This implies that $R:= X- V\in \X_{\mathrm{Reeb}}(M, \H)$ hence $X\in \X_{\mathrm{Reeb}}(M, \H)+ \Gamma(\H)$.
\end{proof}

It is also useful to use the dual point of view on contact structures, i.e. to view $\H$ as the kernel of a 1-form with coefficients in $L$; this can be realized tautologically, by reinterpreting the canonical projection from $TM$ to $L$ as a 1-form
\[ \theta\in \Omega^1(M, L).\]
Note that the curvature $c_{\H}$ can be written as $c_{\H}(X, Y)= \theta([X, Y])$. We say that $\theta$ is of {\bf contact type}. The case in which $L$ is the trivial line bundle gives rise to the standard notion of contact forms \cite{contact}. The previous lemma gives immediately:

\begin{corollary}\label{corollary-1}
$\theta$ restricts to a vector space isomorphism
\begin{equation}\label{Reeb-iso} 
\theta|_{\X_{\mathrm{Reeb}}(M, \H) }: \X_{\mathrm{Reeb}}(M, \H) \stackrel{\sim}{\rmap} \Gamma(L) .
\end{equation}
\end{corollary}

This allows us to transfer the Lie algebra structure of $\X_{\mathrm{Reeb}}(M, \H)$ to a Lie algebra structure on $\Gamma(L)$, denoted by $\{\cdot, \cdot\}$. 

\begin{definition} 
The bracket $\{\cdot, \cdot\}$ on $\Gamma(L)$ is called the {\bf Reeb bracket} associated to the contact manifold $(M, \H)$. 
\end{definition}

Next we rewrite Lemma \ref{lemma-1} in a more convenient form:

\begin{lemma}\label{lemma-2} The map 
\[ \X(M)\cong  \Gamma(L)\oplus \Gamma(\hom(\H, L)), \ \ \ X\mapsto (\theta(X), \theta([\cdot, X])),\]
is an isomorphism of vector spaces; the induced $C^{\infty}(M)$-module structure on the right hand side is given by
\[ f\cdot (u, \phi)= (f u, \phi+ df\otimes u).\]
\end{lemma}

\begin{proof} Lemma \ref{lemma-1} combined with the isomorphism (\ref{Reeb-iso}) and the one between $\H$ and $\hom(\H, L)$ induced by $c_{\H}$ ($V\mapsto c_{\H}(\cdot, V)= \theta([\cdot, V])$) yields the claimed isomorphism of vector spaces. As for the induced $C^{\infty}(M)$-module structure, if 
$(u, \phi)= (\theta(X), \theta([\cdot, X]))$, then 
\[ (\theta(fX), \theta([\cdot, fX]))= (f\theta(X), f\theta([\cdot, X])- df(\cdot)\theta(X))= (fu, f\phi+ df\otimes u),\]
thus completing the proof.
\end{proof}

\begin{remark}\label{Spencer-dec}\rm \  For any vector bundle $E$ over $M$, the bundle of first jets of sections of $E$, denoted by $J^1E$, fits into a short exact sequence of vector bundles over $M$:
\[ 0\rmap \hom(TM, E)\stackrel{i}{\rmap} J^1E \stackrel{pr}{\rmap} E\rmap 0,\]
where $pr$ is the canonical projection and $i$ is determined by 
\[ i(df\otimes u)= f j^1(u)- j^1(fu).\]
Passing to sections, the resulting sequence has a canonical splitting ($u\mapsto j^1(u)$); hence one obtains a decomposition 
\begin{equation}\label{eq:Spencer_decomposition} \Gamma(J^1E)\cong \Gamma( \textrm{Hom}(TM, E))\oplus \Gamma(E) ,\end{equation}
which is henceforth referred to as the {\bf Spencer decomposition}. Note that the induced $C^{\infty}(M)$-module structure on the right hand side is given by precisely the same formula as in Lemma \ref{lemma-2}. The striking similarity between the two is clarified in the statement of Theorem \ref{theorem-1} of section \ref{From contact groupoids to Jacobi manifolds}.
\end{remark}

The notions introduced thus far allow to construct further important geometric objects associated to contact structures. Firstly, surjectivity of (\ref{Reeb-iso}) means that for any section $u\in\Gamma(L)$, there exists a unique vector field $R_u\in\X(M)$ with the property that 
\begin{eqnarray*}
\theta(R_u)=u\qquad\text{and}\qquad \theta([R_u, X])= 0\ \ \textrm{for\ all} \ X\in \Gamma(\H).
\end{eqnarray*}
For $u\in \Gamma(L)$, $R_u$ is called \textbf{the Reeb vector field associated to $u$}. The characterizing property for the Reeb bracket $\{\cdot, \cdot\}$ is
\[  [R_u, R_v]= R_{\{u, v\}}\ \ \ \ \ \ \textrm{for\ all}\ u, v \in \Gamma(L).\]
Applying $\theta$, one obtains the explicit formula
\[ \{u, v\}= \theta([R_u, R_v])\]
relating the Reeb bracket with the 1-form $\theta$ and the Reeb vector fields. Lemma \ref{lemma-2} implies that, for $f\in C^{\infty}(M)$, $u\in \Gamma(L)$,
\[ R_{fu}= fR_{u}+ b(df\otimes u),\]
where 
$b:\hom(\H,L)\to \H$
is the isomorphism induced by $c_{\H}$ (sending $c_{\H}(V, -)\in \hom(\H,L)$ to $V\in \H$). Note that the inverse of the isomorphism defined in Lemma \ref{lemma-2} sends $(u, \phi)$ to $R_u- b(\phi)$. 


\begin{example}\label{example-1}\rm \  When $L$ is the trivial bundle the Reeb vector field associated to the constant function $1$ is the standard Reeb vector field $R$ associated to the contact form 
$\theta$ \cite{contact}; it is uniquely determined by 
\[ \theta(R)= 1, \ i_{R}(d\theta)= 0.\]
The other Reeb vector fields correspond to arbitrary $f\in C^{\infty}(M)$:
\[ R_f= fR+ b(df).\]
Note that, in this case, $b: \H^*\to \H$ is the isomorphism induced by $d\theta$. The Reeb bracket becomes a bracket on $C^{\infty}(M)$. To write down the formula more explicitly, 
one uses $b$ to reinterpret $d\theta|_{\H}$ as an element in $\Lambda^2\H\subset \Lambda^2TM$, i.e. 
as a bivector $\Lambda\in \X^2(M)$. The bracket becomes:
\begin{eqnarray}\label{eq:jacobi-bracket} \{f, g\}= \Lambda(df, dg)+ R(f)g- fR(g).\end{eqnarray}
\end{example}

\section{Jacobi structures and the associated Lie algebroids}
\label{jacobi}

In this section we recall the notion of Jacobi structure, we discuss the associated Lie algebroid and then we conclude with 
the natural representation of the Lie algebroid on the line bundle (to be exploited in the later sections).

As mentioned in the introduction, there are various ways to look at Jacobi structures. We follow here Kirillov \cite{Kirillov}
(who uses the term local Lie algebra) and Marle \cite{Marle} (who uses the term Jacobi bundle). For the equivalence with  
Lichnerowicz's locally conformal Jacobi structures \cite{Lich78} we refer to \cite{Lichnerowicz, Dazord-Lich-Marle91}.

\begin{definition}  A \textbf{Jacobi structure} on a manifold $M$ is a pair $(L, \{\cdot, \cdot \})$ consisting of a line bundle $L \to M$ and a Lie
bracket $\{\cdot, \cdot \}$ on the space of sections $\Gamma(L)$, with the property that it is local in the sense that
\[ \textrm{supp}(\{u, v\})\subset \textrm{supp}(u)\cap \textrm{supp}(v)\ \ \ \ \forall\ u, v\in \Gamma(L).\]
\end{definition}

\begin{example}\label{example-2}\rm When $L$ is the trivial bundle Kirillov proved in \cite{Kirillov} that the Jacobi bracket is 
determined by a pair $(\Lambda, E)$ consisting of a bivector $\Lambda\in\X^2(M)$ and a vector field $R\in\X(M)$, satisfying 
\[[\Lambda,\Lambda]=2R\wedge\Lambda,\ \ [\Lambda,R]=0.\]
Any such pair induces the bracket given by (\ref{eq:jacobi-bracket}) on $\Gamma(L)= C^{\infty}(M)$ (and conversely). 
Such a pair $(\Lambda, R)$ will be called a {\bf Jacobi pair}; they correspond to the Jacobi structures of Lichnerowicz \cite{Lich78}. 
\end{example}

\begin{example} \rm \ 
The previous section shows that any contact structure has an underlying Jacobi structure. Actually, as in the case of symplectic and Poisson structures, contact structures can be seen as ``non-degenerate Jacobi structures''.
\end{example}

Next, we introduce the Lie algebroid associated to a Jacobi structure which was first defined in \cite{Dazord}.  

\begin{proposition}\label{prop:liealgjac} For any Jacobi structure $(L, \{\cdot, \cdot \})$:
\begin{enumerate}
\item\label{item1} There is a unique vector bundle morphism $\rho: J^1L\rmap TM$
such that, for all $u, v\in \Gamma(L)$, $f\in C^{\infty}(M)$,
\[ \{u, fv\}= f\{u, v\}+ L_{\rho(j^1u)}(f) v. \]
\item\label{item2} There is a unique Lie algebroid structure on $J^1L$ with anchor $\rho$ and whose Lie bracket 
$[\cdot, \cdot]$ on $\Gamma(J^1L)$ satisfies 
\begin{eqnarray}\label{eq-jet-lie-bracket}
[j^1u,j^1v]=j^1\{u,v\},\qquad \forall u,v\in\Gamma(L)
\end{eqnarray}
\end{enumerate}
\end{proposition}

Recall that part \ref{item2} means that $[\cdot, \cdot]$ makes $\Gamma(J^1L)$ into a Lie algebra and that it satisfies the Leibniz identity \[ [\alpha, f\beta]= f [\alpha, \beta]+ L_{\rho(\alpha)}(f) \beta,\ \ \ \alpha, \beta\in \Gamma(J^1L), f\in C^{\infty}(M).\]

\begin{proof} For part \ref{item1}, the conditions in the statement can be rewritten using the Spencer decomposition \eqref{eq:Spencer_decomposition} for $\Gamma(J^1L)$. Giving a bundle map $\rho: J^1L\rmap TM$ is equivalent to giving a pair of maps
\[ \rho^1: \Gamma(L)\rmap \X(M), \ \rho^2: \hom(TM, L)\rmap TM,\]
where $\rho^2$ is a vector bundle map, $\rho^1$ is linear and they are related by 
\begin{eqnarray}\label{eq-defining-property}
\rho^1(fu)=f\rho^1(u)-\rho^2(df\otimes u).
\end{eqnarray}
Note that $\rho^1= \rho\circ j^1$; hence the condition in the statement yields the following for $\rho^1$:
\begin{eqnarray}\label{eq-defining-property-rho^1}
\{u,fv\}=f\{u,v\}+L_{\rho^1(u)}(f)v.
\end{eqnarray}
Equations (\ref{eq-defining-property}) and  (\ref{eq-defining-property-rho^1}) can be used to define uniquely $\rho^1$ and $\rho^2$ (hence also $\rho$) as follows. The idea is to use a result of Kirillov \cite{Kirillov} which says that $\{\cdot,\cdot\}$ must be a differential operator of order at most one in each argument. Recall that a differential operator of order at most one $P:\Gamma(E)\to\Gamma(F)$, between sections of vector bundles, has a symbol 
\[ \sigma_P\in \Gamma(TM\otimes \hom(E,F))\]
uniquely determined by the property:
\begin{eqnarray*}
P(fu)=fP(u)+\sigma_P(df)(u),\ \ \forall u\in\Gamma(E),f\in C^{\infty}(M).
\end{eqnarray*}
When $E=F=L$ is a line bundle, $\hom(L,L)$ is trivial and therefore $\sigma_P\in \X(M)$. The defining equations (\ref{eq-defining-property-rho^1}) and (\ref{eq-defining-property}) can be interpreted as saying that $\rho^1(u)$ is the symbol of the operator $\{u, \cdot\}$ and that $\rho^2$ is minus the symbol of $\rho^1$. Hence their existence follows from Kirillov's result; uniqueness is clear. \\

For part \ref{item2}, first observe that the condition on $[\cdot, \cdot]$, the Leibniz identity and the fact that $\Gamma(J^1L)$ is generated as a $C^{\infty}(M)$-module by elements of type $j^1(u)$, imply the uniqueness of the bracket, and also indicate the actual formula for it. To see that the resulting bracket is well-defined, one can for instance write  $[\cdot, \cdot]$ explicitly using the Spencer decomposition \eqref{eq:Spencer_decomposition}; alternatively, formula (\ref{eq-jet-lie-bracket}) above can be taken as the definition of the bracket. Either way, $[\cdot, \cdot]$ clearly 
satisfies the Leibniz identity. To prove the Jacobi identity, first note that $\rho$ induces a Lie algebra map at the level of sections. Indeed, the expression $\rho([\alpha, \beta])- [\rho(\alpha), \rho(\beta)]$ is easily seen to be $C^{\infty}(M)$-linear on $\alpha, \beta\in \Gamma(J^1L)$; hence it may be assumed that $\alpha= j^1u, \beta= j^1v$ with $u, v\in \Gamma(L)$ case in which the expression becomes
\[ \rho^1({u, v})- [\rho^1(u), \rho^1(v)].\]
Recall that $\rho^1(u)$ was the symbol of $P_u:= \{u, \cdot\}$. Also, the Jacobi identity for $\{\cdot, \cdot\}$ means that $P_{\{u, v\}}$ is the commutator $[P_u, P_v]$; hence, passing to symbols, the previous expression vanishes. In conclusion, $\rho$ is a Lie algebra morphism. Using this and the Leibniz identity, a simple computation shows that the Jacobiator of $[\cdot, \cdot]$ is $C^{\infty}(M)$-linear in all arguments. Hence, again, it suffices to check the Jacobi identity on elements of type $j^1(u)$, which follows from the Jacobi identity for $\{\cdot, \cdot\}$. 
\end{proof}

\begin{example}\label{example-2.2}\rm Continuing example \ref{example-2}, i.e. when $L$ is the trivial line bundle and we deal with a Jacobi pair $(\Lambda, R)$, the Lie algebroid $J^1L$ is isomorphic to $T^*M\oplus \mathbb{R}$; working out the Lie bracket one finds the long formulas of \cite{Kerbrat}. 
\end{example}

Next, we show that $L$ has a natural structure of {\bf representation} of the Lie algebroid $J^1L$, i.e. it comes with a bilinear map
\[ \nabla: \Gamma(J^1L)\times \Gamma(L)\rmap \Gamma(L), \ (\alpha, u)\mapsto \nabla_{\alpha}(u),\]
satisfying the usual connection-type identities (see e.g. \cite{CrainicFernandes:lecture}) + the flatness condition 
\[ \nabla_{[u, v]}= \nabla_{u}\nabla_{v}- \nabla_{v}\nabla_{u}.\]
One thinks of $\nabla$ as ``infinitesimal action of $J^1L$ on $L$''. 

The next lemma is proven by arguments similar to (but simpler than) those of proof of Proposition \ref{prop:liealgjac}, part 2.

\begin{lemma}\label{lemma:representation} There is a unique action $\nabla$ of $J^1L$ on $L$ satisfying
\[ \nabla_{j^1(u)}(v)= \{u, v\},\ \ \forall u,v\in\Gamma(L).\]
\end{lemma}

\begin{example}\label{trivial-but-not-1}\rm \ 
When $L$ is trivial (example \ref{example-2.2}) the action is still non-trivial: it actually encodes $R$! Indeed, $\nabla_{j^1f}(1)=-R(f)$. 
\end{example}

\section{Jacobi structures and the associated Spencer operator}

In this section we recall the definition of Spencer operators and indicate their fundamental role in the study of Jacobi structures (Proposition \ref{cor-Jac-Spencer}). Then we move to their global counterpart: multiplicative forms and distributions on 
groupoids \cite{Maria, thesis}. \\

The classical Spencer operator associated to a vector bundle $L$ \cite{Spencer},
\begin{equation}
\label{classical Spencer operator} D: \Gamma(J^1L)\rmap \Omega^1(M, L),\ \ (X,\alpha)\mapsto D_X(\alpha) =D(\alpha)(X), 
\end{equation}
is the canonical projection on the first factor of the Spencer decomposition \eqref{eq:Spencer_decomposition}.

\begin{definition}\label{def:Spencer-operator} Let $A$ be a Lie algebroid over $M$, let $E$ be a representation of $A$ with associated operator denoted by $\nabla$, and let $l:A\to E$ be a surjective vector bundle map. A {\bf Spencer operator (on the Lie algebroid $A$) relative to $l$} is a bilinear operator
\[ D: \X(M)\times \Gamma(A)\rmap \Gamma(E),\ (X, \alpha)\mapsto D_X(\alpha)\]
 which is $C^{\infty}(M)$-linear in $X$, satisfies the Leibniz identity relative to $l$:
\[ D_X(f\alpha)= fD_X(\alpha)+ L_{X}(f) l(\alpha),\]
and the following two compatibility conditions:
\begin{equation}\label{horizontal}D_{\rho(\alpha)}(\alpha')= \nabla_{\alpha'}(l(\alpha))+ l([\alpha, \alpha'])\end{equation}
\begin{equation}\label{vertical}
\qquad D_X[\alpha,\alpha']=\nabla_\alpha(D_X\alpha')-D_{[\rho(\alpha),X]}\alpha'-\nabla_{\alpha'}(D_X\alpha)+D_{[\rho(\alpha'),X]}\alpha,
\end{equation}
for all $\alpha, \alpha'\in \Gamma(A)$, $X\in\mathfrak{X}(M)$.

\end{definition}

It is easy to see that given a Jacobi structure $(L,\{\cdot,\cdot\})$ on $M$, the classical Spencer operator \eqref{classical Spencer operator} becomes a Spencer operator on the Lie algebroid $J^1L$ associated to $(L,\{\cdot,\cdot\})$. Actually this gives a full characterization of Jacobi structures:

\begin{proposition}\label{cor-Jac-Spencer} Given a line bundle $L$ over a manifold $M$, there is a canonical, bijective correspondence between:
\begin{enumerate}
\item Jacobi structures with underlying line bundle $L$,
\item Lie algebroid structures on $J^1L$ with the property that the classical Spencer operator is a Spencer operator on $J^1L$ relative to the canonical projection $ \textrm{pr}: J^1L\rmap L$.  
\end{enumerate}
\end{proposition}

\begin{proof} We still have to show how the Lie algebroid structure on $J^1L$ induces the Jacobi bracket $\{\cdot, \cdot\}$ on $\Gamma(L)$. We simply define
\[ \{u, v\}:= \textrm{pr}([j^1u, j^1v]).\]
Clearly this is antisymmetric and local in $u$ and $v$. 
Since $D$ vanishes precisely on holonomic sections of $J^1L$ (i.e. of type $j^1u$), equation  (\ref{vertical}) implies that all the expressions of type $[j^1u, j^1v]$ must be holonomic, hence
\[ [j^1u, j^1v]= j^1(\textrm{pr}([j^1u, j^1v]))= j^1\{u, v\}.\]
The Jacobi identity for $\{\cdot, \cdot\}$ follows from that of $[\cdot, \cdot]$. Moreover, comparing the formulas, we see that the two constructions are inverse to each other. 
\end{proof} 

\vspace*{.2in}

Next we look at the global counterpart of Spencer operators (to be applied in Section \ref{sec:jactoctc} to obtain the contact groupoids integrating Jacobi structures). We briefly recall some terminology on Lie groupoids \cite{CrainicFernandes:lecture,Mackenzie:General}. We fix a Lie groupoid $\Sigma$ over $M$; recall that $\Sigma$ denotes the manifold of arrows and $M$ the manifold of objects, $s, t:\Sigma\to M$ denote the source and the target map, respectively, and $m(g, h)= gh$ the multiplication. The right translation $r_g$ induced by an arrow $g: x\rmap y$ is a diffeomorphism from $s^{-1}(y)$ to $s^{-1}(x)$; by differentiation, it induces:
\begin{equation}\label{inf-right-transl} 
r_g:   T^{s}_{a}\Sigma \rmap T^{s}_{ag}\Sigma ,
\end{equation}
where $T^s\Sigma= \textrm{Ker}(ds)$ stands for the bundle of vectors tangent to the $s$-fibers. Recall that the Lie algebroid $A= A(\Sigma)$ associated to $\Sigma$ is, as a vector bundle over $M$, the restriction of $T^{s}\Sigma$ to $M$, where $M$ sits inside $\Sigma$ as units. Using right translations, any $\alpha\in \Gamma(A)$ induces a right invariant vector field (tangent to the $s$-fibers), $\alpha^{r}\in \X^{\mathrm{inv}}(\Sigma)$:
\begin{equation}\label{alpha-r} 
\alpha^r(g)= r_g(\alpha_{t(g)})). 
\end{equation}
This induces an isomorphism (and then the Lie bracket on $\Gamma(A)$):
\[ \Gamma(A) \stackrel{\sim}{\to} \X^{\mathrm{inv}}(\Sigma), \ \ \alpha\mapsto \alpha^r.\]

To discuss multiplicative structures, we use the Lie groupoid $T\Sigma$ over $TM$ whose structure maps are just the differentials of the structure maps of $\Sigma$.

\begin{definition}\label{def:multiplicative_distributions} A (constant rank, smooth) distribution $\H\subset T\Sigma$ is called {\bf multiplicative} if $\H$ is a Lie subgroupoid of $T\Sigma$ with the same base $TM$.
\end{definition}

For a multiplicative distribution $\H$ one defines its $s$-vertical part:
\[ \H^{s}:= \H\cap T^s\Sigma, \] 
Note that, since $\H$ is multiplicative, $0_g\in \H$ and $r_g(X_a)= (dm)_{a, g}(X_a, 0_g)$ for all $X_a\in T^{s}_{a}\Sigma$,
it follows that $\mathcal{H}^{s}$ is invariant under the right translations (\ref{inf-right-transl}):
\begin{equation}\label{item4}
r_g(\mathcal{H}^{s}_{a})= \mathcal{H}^{s}_{ag}.
\end{equation}
Also, since $ds: \H\to TM$ is surjective, for every $X\in T\Sigma$ one finds $V\in \H$ such that $ds(X)= ds(V)$, therefore $\H$ is transversal to the $s$-fibers:
\begin{equation}\label{item3}
T\Sigma= T^s\Sigma+ \H .
\end{equation}
Similar statements arise using left translations acting on the spaces $T^t\Sigma$ and $\H^{t}$.


One also has a dual point of view on multiplicative distributions, obtained by using forms, when (as in the case of contact structures) one reinterprets the projection modulo $\H$ as a 1-form. However, in this setting, the quotient line bundle $\tilde L:=T\Sigma/\H$ is determined by its restriction to $M$:
\begin{eqnarray}\label{restriction-line-bundle}
L:=\tilde L|_M .
\end{eqnarray} 
Indeed, (\ref{item3}) shows that $\tilde{L}= T^{s}\Sigma/\H^{s}$ and then (\ref{item4}) implies that the right translations induce isomorphisms $r_{g}: \tilde{L}_a\to \tilde{L}_{ag}$ whenever $ag$ is defined, in particular, $r_{g}: L_{t(g)}\to \tilde{L}_g$. In fact:

\begin{lemma}\label{eq:right-translation} $L$ is a representation of $\Sigma$ and the right translations induce an isomorphism of vector bundles over $\Sigma$
\begin{eqnarray*} r: t^*L\stackrel{\sim}{\rmap} \tilde{L},\ \ t^*u\mapsto u^r.\end{eqnarray*}
\end{lemma}

\begin{proof} For any arrow $g:x\to y$, right translations induce a map $r_g: L_y\stackrel{\sim}{\rmap} \tilde{L}_g$; similarly, one has that left translations induce an isomorphism $l_g: L_x \stackrel{\sim}{\rmap} \tilde{L}_g$. Combining the two, one obtains that $g$ induces an isomorphism $L_x\rmap L_y,v\mapsto g\cdot v$, which satisfy the usual identity for an action and vary smoothly w.r.t $g$ and $v$.
\end{proof}

Henceforth, the canonical projection modulo $\H$ is interpreted as a 1-form
\begin{equation}\label{eq:one_form} \theta\in \Omega^1(\Sigma, t^*L) .\end{equation}
For forms with values in a representation, one can talk about their multiplicativity:

\begin{definition} Let $\Sigma$ be a Lie groupoid and $E$ a representation of $\Sigma$. An $E$-valued {\bf multiplicaitve} one form is any form 
$\eta\in \Omega^1(\Sigma;t^*E)$ satisfying 
\begin{equation}
(m^{\ast}\eta)_{(g,h)} = \pr_1^{\ast}\eta + g\cdot(\pr_2^{\ast}\eta),
\end{equation}
for all $(g,h)$ in the domain $\Sigma_2$ of the multiplication $m$, where $\pr_1,\pr_2: \Sigma_2 \to \Sigma$ denote the canonical projections. We say that $\eta$ is regular if it is surjective. 
\end{definition}

It is clear that the kernel of any regular multiplicative form is a a multiplicative distribution. Conversely, a rather straightforward computation shows that the form \eqref{eq:one_form} associated to a multiplicative distribution is multiplicative.\\


Returning to the infitesimal picture, the key remark is that any multiplicative form on a groupoid induces a Spencer operator (see Definition \ref{def:Spencer-operator}) on the Lie algebroid of the groupoid: 

\begin{proposition}\label{prop:integrability_multiplicative_distributions} Let $\Sigma$ be a Lie groupoid over $M$ with Lie algebroid $A$. Then any multiplicative distribution $\H\subset T\Sigma$ induces:
\begin{itemize}
 \item the vector bundle $L$ which is the restriction to $M$ of $T^s\Sigma/\H^s$;
 \item the vector bundle morphism 
\[ l: A\rmap L, \ \ \ l(\al)=\al\ \mod\ \H^s ;\]
 \item a Spencer operator $D$ on the Lie algebroid $A$ relative to $l$:
\begin{eqnarray}\label{Spencer-operator}
D_X(\al)=[\tilde X,\al^r]|_M\ \mod \ \H^s.
\end{eqnarray}
Here, for $X\in\X(M)$, $\tilde X\in\Gamma(H)$ is any extension of $X$ to $\Sigma$ (where $M\overset{u}{\hookrightarrow}\Sigma$ as units) with the property that $d_gs(\tilde X_g)=X_{s(g)}$, and $\alpha^r$ is given by \eqref{alpha-r} .
\end{itemize}
\end{proposition}

This follows by a lengthy but straightforward computation \cite{thesis, Maria}. Its converse, an integrability theorem (for Spencer operators), is less straightforward and is one of the main results of \cite{thesis, Maria}:

\begin{theorem}\label{thm:integrability_multiplicative_distributions} For $\Sigma$ s-simply connected one has a 1-1 correspondence between 
\begin{itemize}
\item multiplicative distributions $\H$ on $\Sigma$;
\item Spencer operators on the Lie algebroid $A$ relative to some map $l: A\twoheadrightarrow E$.
\end{itemize}
The correspondence is given by \eqref{Spencer-operator}.
\end{theorem}

The relevant Spencer operator of a Jacobi structure $(L,\{\cdot,\cdot\})$ on $M$ is the classical Spencer operator associated to the line 
bundle $L$. Theorem \ref{theorem-2} of section \ref{jacobi} states that the multiplicative distribution 
$\H\subset T\Sigma$ integrating the classical Spencer operator makes the pair $(\Sigma,\H)$ into a {\it contact groupoid} in the sense of the next section.

\section{Contact groupoids}
\label{Contact groupoids}

In this section we recall the notion of contact groupoid \cite{Dazord1, Dazord}, the dual point of view (using multiplicative forms) and we discuss the first consequences of the compatibility of the contact structure with the groupoid structure.

\begin{definition} A {\bf contact groupoid} over $M$ is a pair $(\Sigma,\H)$ consisting of a Lie groupoid $\Sigma$ over $M$ and contact structure $\H$ on $\Sigma$, with the property that $\H$ is multiplicative in the sense of definition \ref{def:multiplicative_distributions}.
\end{definition}

Using the previous section, since $\H$ is multiplicative:
\begin{itemize}
\item the line bundle $T\Sigma/\H$ is determined by its restriction $L$ to $M$. Moreover $L$ is a representation of $\Sigma$;
\item contact groupoids can also be described as groupoids $\Sigma$ endowed with a one dimensional representation $L$ and a multiplicative one form $\theta\in \Omega^1(\Sigma, t^*L)$
which is of contact type.
\end{itemize}

\begin{example}\label{trivial-but-not-2}\rm \ 
When $L$ is trivial as a line bundle, the action of $\Sigma$ on $L$ may still be non-trivial and it will be encoded in a 1-cocycle $r$ on $\Sigma$. Hence, in this case, the structure of contact groupoid is encoded in a contact form $\theta$ and a 1-cocycle $r$; working out the multiplicativity conditions, we find the (rather puzzling) equation that is taken as an axiom in \cite{Kerbrat}.
\end{example}

Recall that the {\bf contact orthogonal}  $\mathcal{F}_x^c\subset \H_x$ of a subspace $\mathcal{F}_x\subset\H_x$ is the orthogonal w.r.t. the non-degenerate pairing \eqref{eq:curvature} induced by $\H$. A submanifold $N\subset M$ is {\bf Legendrian} if $TN\subset\H|_N$ and $(T_xN)^c=T_xN$, for all $x\in N.$

\begin{proposition}\label{lemma:contact-groupoid} In a contact groupoid $(\Sigma, \H)$ over $M$:
\begin{enumerate}
\item\label{itemi} The unit map $u: M\hookrightarrow \Sigma$ is a Legendrian embedding.
\item\label{itemii} $\H^{s}= (\H^{t})^{c}$.
\end{enumerate}
\end{proposition}

Denote by $\X_{\mathrm{Reeb}}^{\mathrm{inv}}(\Sigma,\H)\subset \X(\Sigma)$ the subpace defined by
$$\X_{\mathrm{Reeb}}^{\mathrm{inv}}(\Sigma,\H):=\X_{\mathrm{Reeb}}(\Sigma,\H)\cap \X^{\mathrm{inv}}(\Sigma).$$
Multiplicativity of $\theta$, Corollary \ref{corollary-1} and Lemma \ref{eq:right-translation} imply the following result.

\begin{corollary}\label{corollary:right-invariant} The isomorphism
$\theta: \X_{\mathrm{Reeb}}(\Sigma,\H) \stackrel{\sim}{\rmap} \Gamma(t^*L)$ of Corollary \ref{corollary-1} is invariant under right translation \eqref{inf-right-transl}, i.e. $\theta(r(v))=\theta(v)$. In particular, $\theta$ restricts to a vector space isomorphism
$$\theta|_{\X_{\mathrm{Reeb}}^{\mathrm{inv}}(\Sigma,\H)}: \X_{\mathrm{Reeb}}^{\mathrm{inv}}(\Sigma,\H) \stackrel{\sim}{\rmap} \Gamma(L),$$
with inverse $u\mapsto R_{u^r}.$
\end{corollary}

To prove Proposition \ref{lemma:contact-groupoid}, we need the following:

\begin{lemma}\label{lemma:invariance-lie-bracket}
Let $\theta\in\Omega^1(\Sigma,t^*E)$ be a multiplicative form with values in some representation $E$ of $\Sigma$. Then, for any $X\in\Gamma(\ker\theta)$ and any $\alpha^r\in\X^{\mathrm{inv}}(\Sigma)$,
\begin{eqnarray}\label{invariance-bracket}
\theta_g([\alpha^r,X])=\theta_{t(g)}([\alpha^r,\tilde X]),\end{eqnarray}
where $\tilde X\in\Gamma(\ker\theta)$ is any other vector field with the property that $dt(\tilde X)=dt(X)$.
\end{lemma}

\begin{proof}
For any integer $k$ and any section $\alpha$ of the Lie algebroid of $\Sigma$, consider 
\[\Lie_{\alpha}:\Omega^k(\Sigma,t^*E)\to \Sigma^k(\G,s^*E),\] 
\[(\Lie_{\alpha}\theta)_g:=\frac{d}{d\epsilon}\big{|}_{\epsilon=0}(\varphi^{\epsilon}_{\al^r}(g))^{-1}\cdot(\varphi^\epsilon_{\al^r})^*\theta|_{\varphi^\epsilon_{\al^r}(g)}\]
where $\varphi_{\al^r}^\epsilon:\Sigma\to\Sigma$ is the flow of $\al^r$. In general (cf. e.g. Lemma 3.8 of \cite{Maria} or \cite{thesis}), 
\[[i_X,L_\al](\theta)_g=g^{-1}\cdot\theta_g([X,\al^r\, ]).\]
\[\theta_g([X,\al^r])=g\cdot L_\al(\theta)_g(X).
\]

Hence, to prove (\ref{invariance-bracket}), it suffices to show that the last expression does not depend on $g$ and $X$, but only on $t(g)$ and $dt(X)$. For that, we remark that: 
\[g\cdot L_\al(\theta)_g(X)=\frac{d}{d\epsilon}\big{|}_{\epsilon=0}(\varphi^{\epsilon}_{\al}(t(g)))^{-1}\cdot \theta(d\varphi^\epsilon_\al(dt(X))),\] 
where $\varphi_\al:M\to\Sigma$ is the restriction of $\varphi_{\al^r}$ to $M$. Indeed, 
\begin{eqnarray*}\begin{split}
g\cdot (L_\al\theta)(X) &=g\cdot\frac{d}{d\epsilon}\big{|}_{\epsilon=0}(\varphi^{\epsilon}_{\al^r}(g))^{-1}\cdot(\varphi^\epsilon_{\al^r})^*\theta|_{\varphi^\epsilon_{\al^r}(g)}(X)\\
&=g\cdot\frac{d}{d\epsilon}\big{|}_{\epsilon=0}(\varphi_{\al}^\epsilon(t(g))\cdot g)^{-1}\cdot (\theta(dm(d\varphi^\epsilon_\al(dt(X)),X))\\
&=g\cdot\frac{d}{d\epsilon}\big{|}_{\epsilon=0}g^{-1}\cdot(\varphi^{\epsilon}_{\al}(t(g)))^{-1}\cdot (\theta (dm(d\varphi^\epsilon_\al(dt(X)),X))\\
&=\frac{d}{d\epsilon}\big{|}_{\epsilon=0}(\varphi^{\epsilon}_{\al}(t(g)))^{-1}\cdot (\theta(d\varphi^\epsilon_\al(dt(X)))-\varphi_{\al}^\epsilon(t(g))\cdot\theta(X))\\
&=\frac{d}{d\epsilon}\big{|}_{\epsilon=0}(\varphi^{\epsilon}_{\al}(t(g)))^{-1}\cdot \theta(d\varphi^\epsilon_\al(dt(X))),\end{split}
\end{eqnarray*}
where it is used that the flow of a right invariant vector field $\al^r$ is given by $\varphi_{\al^r}^\epsilon(g)=\varphi_{\al}^\epsilon(t(g))\cdot g$ and therefore for a fixed $\epsilon$, $d\varphi_{\al^r}^\epsilon=dm(d\varphi_{\al}^\epsilon\circ dt,id)$. 
\end{proof}

\begin{proof}[Proof of Proposition \ref{lemma:contact-groupoid}] For item \ref{itemi}, to show that $TM\subset TM^c$ let $X,Y\in\X(M),$ and choose any $t$-projectable extensions $\tilde X,\tilde Y\in\Gamma(\H)$ of $X,Y$ respectively 
with the property 
$\tilde X|_M=X,\tilde Y|_M=Y.
$
This can  be done because for multiplicative distributions $\H$ the restriction of $dt:T\Sigma\to TM$ to $\H$ is surjective. With this 
$$c_\H(X,Y)(1_x)=[\tilde X,\tilde Y](1_x)\mod \H=[X,Y](x)\mod \H=0,$$
where the last equality holds since $[X,Y]\in TM\subset \H$. 
For the other inclusion split $\H_x$, $x\in M$, as the direct sum $T_xM\oplus \H^s_x$ and prove that $\al_x\in T_xM^c$ only if $\al_x\in T_xM$. If not, without loss of generality $\al_x\in \H^s$, and Lemma \ref{lemma:invariance-lie-bracket} shows that
\[
c_\H(\al_x,X_x)=c_\H(\al_x,dt(X))=0,\ \ \forall X\in\H,
\]
which only happens if $\al_x=0$ as $c_\H$ is non-degenerate. This last equation shows that $TM^c=TM$. A similar computation shows that $\H^t\subset (\H^s)^c$. For the other inclusion note that as $TM$ is Legendrian then $2\rank TM=\rank \H$ on the one hand, and on the other $\rank\H=\rank TM+\rank \H^s=\rank TM+\rank \H^t$. A simple dimension count shows that $\H^t=(\H^s)^c$.
\end{proof}

\section{From contact groupoids to Jacobi manifolds}\label{From contact groupoids to Jacobi manifolds}

In this section we
explain the relevance of Spencer operators to the study of Jacobi structures (relevance that was already indicated in 
Proposition \ref{cor-Jac-Spencer}). For clarity, recall that for a contact groupoid $(\Sigma,\H)$ we have:
\begin{itemize}
\item the (normal) line bundle $\tilde L$ of the contact structure, $\tilde L=T\Sigma/\H;$
\item the restriction $L$ of $\tilde L$ to $M$, which is a representation of $\Sigma$;
\item the vector bundle isomorphism $r:t^*L\to\tilde L,t^*(u)\mapsto u^r$ induced by right translations \eqref{inf-right-transl} on $\Sigma$ (Lemma \ref{eq:right-translation});
\item the isomorphism $\theta:\X^{\mathrm{inv}}_{\mathrm{Reeb}}(\Sigma, \H)\to \Gamma(L)$, whose inverse associates to $u$ the Reeb vector field $R_{u^r}$ of the coresponding $u^r\in \Gamma(\tilde L)$
(Corollary \ref{corollary:right-invariant});
\item the Lie algebroid $A$ of $\Sigma$ and the  Spencer operator $D:\Gamma(A)\to \Omega^1(M,L)$ associated to $\H$ (Proposition \ref{prop:integrability_multiplicative_distributions}).
\end{itemize}

\begin{theoremz}\label{theorem-1} Let $(\Sigma,\H)$ be a contact groupoid over $M$. Then:
\begin{enumerate}
\item[1.] there exists a unique Jacobi structure $(L,\{\cdot,\cdot\})$ over $M$ with the property that the target map $t:\Sigma\to M$ is a Jacobi map with bundle component $r:t^*L\simeq\tilde L$,
\item[2.] the Lie algebroid $A$ of $\Sigma$ is isomorphic to the Lie algebroid associated to $(M,L,\{\cdot,\cdot\})$, via the Lie algebroid isomorphism
\[\Phi:J^1L\to A,\quad \Phi(j^1u)= R_{u^r}|_M,\]
\item[3.] after the identification of $A$ with $J^1L$, the Spencer operator $D$ associated to $\H$ becomes the classical Spencer operator \eqref{classical Spencer operator}.
\end{enumerate}
\end{theoremz}

In the previous statement, a map $\phi:(N,\bar{L}) \to (M,L)$ between Jacobi manifolds is said to be {\bf Jacobi} with bundle component $F: \phi^*L\to\bar{L}$, if $F$ is a vector bundle isomorphism and $\phi:\Gamma(L)\to \Gamma(\bar{ L}), u\mapsto F\circ\phi^*u$ is a Lie algebra map. \\

Point 3 combined with the fact that Jacobi structures are encoded in Lie algebroid structures on $J^1L$ for which the classical Spencer operator is a Spencer operator (Proposition \ref{cor-Jac-Spencer}) reveal the appearance of the Jacobi structure $(L, \{\cdot, \cdot\})$.

\begin{proof}
We first show that $A$ is isomorphic to $J^1L$. Recall from Lemma \ref{lemma-1} that $\X(\Sigma)$ can be written as the direct sum 
$\X_{\mathrm{Reeb}}(\Sigma,\H)\oplus\Gamma(\H)$. When restricting this identification to right invariant vector fields, we obtain that 
$\X^{\mathrm{inv}}(\Sigma)=\X_{\mathrm{Reeb}}^{\mathrm{inv}}(\Sigma,\H)\oplus \Gamma^{\mathrm{inv}}(\H), \Gamma^{\mathrm{inv}}(\H):=\Gamma(\H)\cap \X^{\mathrm{inv}}(\Sigma)$. With this, we claim that the right translations 
induce an identification
\begin{equation}\label{decomposition2}
\Gamma(A)\simeq\Gamma(L)\oplus\Gamma(\H^s|_M)
\simeq \Gamma(L)\oplus \Omega^1(M,L).
\end{equation}
The first identification uses Corollary \ref{corollary:right-invariant} and the fact that right translation gives the identification $
\Gamma^{\mathrm{inv}}(\H)\simeq \Gamma(\H^s|_M)$. For the second identification, we use the identification $dt:(\H/\H^t)|_M\simeq TM$ (since $\H$ is multiplicative hence $t$-transversal) and note that the non-degenerate curvature map \eqref{eq:curvature} induces an isomorphism 
\begin{equation}\label{non-deg}
\H^s|_M\to\hom (TM,L),\ \ \ V\mapsto c_\H(V,\cdot).
\end{equation} 
That \eqref{non-deg} is an isomorphism is a consequence of $(\H^s)^c=\H^t$ (Proposition \ref{lemma:contact-groupoid}). 
The decomposition \eqref{decomposition2} can be shown to hold directly as follows. That the sum is direct is clear. 
Let $\al\in\Gamma(A)$ arbitrary, and consider the Spencer operator $D$ associated to $\H$ 
(as in Proposition \ref{prop:integrability_multiplicative_distributions}). Non-degeneracy of \eqref{non-deg} implies that there exists
 $V\in\Gamma(\H^s|_M)$ such that the map $D(\al):TM\to L$ coincides with $c_\H(V,-):TM\to L$. As a consequence, 
$R^r:=\al^r-V^r\in\X_{\mathrm{Reeb}}^{\mathrm{inv}}(\Sigma,\H)$: indeed, by Lemma 
\ref{lemma:invariance-lie-bracket}, for any $X\in \Gamma(\H)$,
\[\theta_g([R^r,X])=\theta_{t(g)}([R^r,\tilde X])=c_\H(V,dt(X))-D_{dt(X)}(\al)=0\]
where $\tilde X\in\Gamma(\H)$ is any $s$-projectable vector extending $u_*(dt(X))$. This implies that $\al^r\in\X_{\mathrm{Reeb}}^{\mathrm{inv}}(\Sigma,\H)+\Gamma^{\mathrm{inv}}(\H)$ hence
 $\al\in\Gamma(L)+\Gamma(\H^s|_M)$. Note that this also shows that $\al\in\Gamma(A)$ belongs to 
$\Gamma(L)$ if and only if $D(\al)=0$. Arguments similar to those in the proof of Lemma \ref{lemma-2} show that the $C^\infty(M)$-structure of 
$\Gamma(A)$ is the one given by the Spencer decomposition \eqref{eq:Spencer_decomposition} of $\Gamma(J^1L)$, which implies that $A$ is isomorphic to $J^1L$ as vector bundles. Hence, $A$ induces a Lie algebroid structure on $J^1L$. \\

As for the proof of item 3, note that $\H^s|_M$ is identified via the map \eqref{non-deg} with the linear subspace $T^*M\otimes L\subset J^1L$ 
and therefore, the quotient map $\theta|_A:A\to L$ is identified with the projection map $pr:J^1L\to L$. On the other hand, and having in mind the identification \eqref{decomposition2}, for
$(u,\omega)\in\Gamma(L)\oplus\Omega^1(M,L)\simeq \Gamma(J^1L)$, and $D$ the Spencer operator associated to $\H$, one obtains that 
\begin{eqnarray*}
\begin{split}
D_X(u,\omega)&=\theta([\tilde X,R_{u^r}+c_\H^{-1}(\omega)^r]|_M)\\
&=\theta([\tilde X,c_\H^{-1}(\omega)^r]|_M)=c_\H(X,c_\H^{-1}(\omega))=\omega(X),
\end{split}
\end{eqnarray*}
where the second equality uses the fact that $[\X_{\mathrm{Reeb}}(\Sigma,\H),\H]\subset \H$. This shows that $D$ coincides with the classical Spencer operator \eqref{classical Spencer operator}. With this and Proposition \ref{cor-Jac-Spencer}, we get
the desired Jacobi bracket:
\[\{u,v\}:=pr([j^1u,j^1v])=\theta([R_{u^r},R_{v^r}]|_M),\ \ \ u,v\in\Gamma(L) \]
where, in the second equality, $R_{u^r}\equiv u\in\Gamma(L)\subset\Gamma(A)$ corresponds to $j^1u\in\Gamma(J^1L)$ in the Spencer decomposition \eqref{eq:Spencer_decomposition}. This concludes the proof of item 2. \\

To conclude the proof of item 1, note that the map $\Gamma(L)\to \Gamma(\tilde L), r:u\mapsto u^r$ is a Lie algebra map. This is clear as 
\[\{u^r,v^r\}_\Sigma=[R_{u^r},R_{v^r}]\mod \H=\{u,v\}^r.\] 
Uniqueness follows from injectivity of the map $r:u\mapsto u^r$, because this implies that there exists a unique bracket on $\Gamma(L)$ making $r$ a Lie algebra morphism.
\end{proof}

\begin{remark}\rm For the Spencer operator $D$ associated to the contact distribution $\H$, denote by $\Gamma(A,D)\subset \Gamma(A)$ the space of sections $\alpha$ with the property that $D(\alpha)=0$. The previous proof shows that a section $\alpha$ belongs to $\Gamma(A,D)$ iff $\alpha^r$ belongs to $\X^{\mathrm{inv}}_{\mathrm{Reeb}}(\Sigma,\H)$. Moreover, $\Gamma(A,D)$ is a Lie subalgebra of $\Gamma(A)$, the map
$\theta:\Gamma(A,D)\to \Gamma(L)$ from Corollary \ref{corollary:right-invariant}
is a Lie algebra isomorphism, and $\Gamma(A)$ can be written as the direct sum
\[\Gamma(A)\simeq \Gamma(A,D)\oplus \Gamma(\H^s|_M).\]
\end{remark}

\section{From Jacobi manifolds to contact groupoids}
\label{sec:jactoctc}

Finally, we discuss the integrability of Jacobi manifolds. Again, the main result was known in the case of trivial line bundles \cite{Jacobi} but, even then, the approach was very computational and indirect (via Poissonization). We urge the reader to compare this section with \cite{Jacobi}. And here is the main result:

\begin{theoremz}\label{theorem-2}
Let $(M,L,\{\cdot,\cdot\})$ be a Jacobi manifold. If the associated Lie algebroid $J^1L$ is integrable, then the source 1-connected groupoid $\Sigma$ integrating $J^1L$ has a unique multiplicative distribution $\H\subset T\Sigma$ with the properties that 
\begin{enumerate}
\item $(\Sigma,\H)$ is a contact groupoid,
\item the Jacobi structure induced by $(\Sigma, \H)$ on $M$ (cf. Theorem \ref{theorem-1}) coincides with the original Jacobi structure. 
\end{enumerate}
\end{theoremz}

Combining Theorems \ref{theorem-1} and \ref{theorem-2}, one concludes that Jacobi structures on a manifold $M$ whose associated Lie algebroid $J^1L$ is integrable are in 1-1 correspondence with contact groupoids with source 1-connected fibers.



\begin{proof}[Proof of Theorem \ref{theorem-2}]
Let $\Sigma$ be the s-simply connected Lie groupoid integrating $J^1L$.
Using Theorem \ref{thm:integrability_multiplicative_distributions}, for proving that the Lie groupoid integrating $J^1L$ is contact, 
it suffices to show that the multiplicative distribution $\H\subset T\Sigma$, whose corresponding Spencer operator $D$ is the classical Spencer operator 
\eqref{classical Spencer operator}, is contact. That $\H$ is of codimension 1 is clear as it is transversal to the s-fibers (equation \eqref{item3}) 
and $L=T\Sigma^s/\H^s|_M$ is one dimensional. To prove that $\H$ is maximally non-integrable, note that as the map $l:J^1L\to L$ from Proposition
 \ref{prop:integrability_multiplicative_distributions} is the projection map, then
\[\H^s|_M=\ker (pr:J^1L\to L)=\hom(TM,L).\]  
With this, if $\al^r\in\Gamma^{\mathrm{inv}}(\H)$, by Lemma \ref{lemma:invariance-lie-bracket}
\[[\al^r,X]_g\mod \H=[\al^r,X]_{t(g)}\mod \H=D_{dt(X)}(\al)(t(g)),\]
for $X\in\Gamma(\H)$ any s-projectable vector field extending $u_*(dt(X))$, and $g\in\Sigma.$ Because $D$ is just the projection of $\Gamma(J^1L)$ 
to $\Omega^1(M,L)$ on the Spencer decomposition \eqref{eq:Spencer_decomposition}, and $ds,dt:\H\to TM$ are fiber-wise surjective (equation \eqref{item3}),
 then for $g\in\Sigma$ on which $\al^r(g)\neq 0$ (hence $0\neq\al:T_{t(g)}M\to L_{t(g)}$), one can always find $X$ so that  
\[[\al^r,X]_g\mod \H=D_{dt(X)}(\al)(t(g))=\al(dt(X_g))\neq 0.\]
This proves that $(\Sigma,\H)$ is a contact groupoid. \\

To show the second part of the theorem, denote by $\{\{\cdot,\cdot\}\}$ the Jacobi bracket induced by the contact 
groupoid $(\Sigma,\H)$. By the proof of Theorem \ref{theorem-1}, one has that $$\{\{u,v\}\}=pr([j^1u,j^1v])$$
for any $u, v\in\Gamma(L)$. On the other hand, formula \eqref{horizontal} for the representation $\nabla$ in terms of the Spencer operator $D$ says that 
$\nabla_{j^1u}(v)=pr([j^1u,j^1v]),$
and Lemma \ref{lemma:representation} writes it as $\nabla_{j^1u}(v)= \{u, v\}$. Therefore $\{\cdot,\cdot\}=\{\{\cdot,\cdot\}\}$.

To conclude the proof of the Theorem, it remains to show the uniqueness of $\H$. But this is immediate by item 3  \ref{theorem-1} and Theorem \ref{thm:integrability_multiplicative_distributions}, as the Spencer operator associated to such an $\H$ must be the classical Spencer operator. 
\end{proof}

\bibliographystyle{abbrv}
\bibliography{contact_isotropic_realisations/bibliography_isotropic_realisations}
\end{document}